\documentclass{amsart}
\usepackage{amsfonts,amssymb,amscd,amsmath,enumerate,verbatim,calc,amsthm}
\usepackage[all]{xy}
\usepackage{euscript}

\newcommand{\wrt}{with respect to}

\newcommand{\fp}{\mathfrak{p}}

\newcommand{\nsl}{\textrm{for \textit{n} sufficiently large}}

\newcommand{\NN}{\mathbb{N}_0 }
\newcommand{\Nd}{\mathbb{N}_0^d}
\newcommand{\Ns}{\mathbb{N}_0^s}
\newcommand{\ZZ}{\mathbb{Z} }

\newcommand{\R}{\mathcal{R}}

\newcommand{\rt}{\rightarrow}
\newcommand{\xar}{\longrightarrow}
\newcommand{\ov}{\overline}

\newcommand{\J}{J_1, \ldots, J_d }

\newcommand{\Min}{\operatorname{Min}}
\newcommand{\Ass}{\operatorname{Ass}}
\newcommand{\grade}{\operatorname{grade}}

\theoremstyle{plain}

\newtheorem{theorem}{Theorem}[section]
\newtheorem{corollary}[theorem]{Corollary}
\newtheorem{lemma}[theorem]{Lemma}
\newtheorem{proposition}[theorem]{Proposition}

\theoremstyle{definition}
\newtheorem{definition}[theorem]{Definition}
\newtheorem{s}[theorem]{}
\newtheorem{remark}[theorem]{Remark}
\newtheorem{example}[theorem]{Example}

\theoremstyle{remark}

\begin{document}

\title{Quasi-finite modules and asymptotic prime divisors}

 \author{Daniel Katz}
\address{Department of Mathematics, University of Kansas, Lawrence, KS 66045}
\email{dlk@math.ku.edu}

\author{Tony J. Puthenpurakal}
\address{Department of Mathematics, Indian Institute of Technology Bombay, Powai, Mumbai 400 076, India}
 \email{tputhen@math.iitb.ac.in}
\date{\today}


\begin{abstract}
Let $A$ be a Noetherian ring, $J\subseteq A$ an ideal and $C$ a finitely generated $A$-module.
In this note we would like to prove the following statement. Let
$\{I_n\}_{n\geq 0}$ be a collection of ideals satisfying : (i) $I_n\supseteq J^n$, for all $n$,  (ii) $J^s\cdot I_s \subseteq I_{r+s}$, for all
$r,s\geq 0$ and (iii) $I_n\subseteq I_m$, whenever $m\leq n$. Then
$\Ass_A(I_nC/J^nC)$ is independent of $n$, for $n$ sufficiently large. Note that the set of prime ideals $\cup_{n\geq 1} \Ass_A(I_nC/J^nC)$
is finite, so the issue at hand is the realization that the primes in
$\Ass_A(I_nC/J^nC)$ \textit{do not} behave periodically, as one might have expected, say
if $\bigoplus _{n\geq 0}I_n$ were a Noetherian $A$-algebra generated in degrees greater than
one. We also give a multigraded version of our results.

\end{abstract}

\maketitle

\section{Introduction}
Let $A$ be a Noetherian ring, $J\subseteq A$ an ideal and $C$ a finitely generated $A$-module.
Then, by a well known theorem of Brodmann (see \cite{B}), $\Ass_A(C/J^nC)$ is a stable set
of prime ideals for $n$ large. Brodmann's result has many applications and has been generalized
in various forms. For example, see \cite{KMR}, \cite{KS}, \cite{EW}, \cite{KW}, \cite{KR} and \cite{H}, among others.

In this note we are motivated by the following question. Given an ideal $J\subseteq A$, a finitely generated $A$-module $C$ and a filtration of ideals $\{I_n\}_{n\geq 0}$ with $J^n\subseteq I_n$ for all $n$, when is the set of associated primes $\Ass_A(I_nC/J^nC)$ a stable set of prime ideals ? It turns out that the desired stabilty holds under very mild conditions on the filtration $\{I_n\}_{n\geq 0}$.  It is important to note that the set of prime ideals $\cup_{n\geq 1} \Ass_A(I_nC/J^nC)$
is well known to be a finite set and the issue at hand is the realization that the primes in
$\Ass_A(I_nC/J^nC)$ \textit{do not} behave periodically, as one might have expected, say
if $\bigoplus _{n\geq 0}I_n$ were a Noetherian $A$-algebra generated in degrees greater than
one (see \cite{EW}).

Not surprisingly, our approach is through graded modules defined over finitely generated $A$-algebras. To elaborate,
let $R = \bigoplus_{n \geq 0}R_n$ be a finitely standard graded $A$-algebra, i.e., $R = A[R_1]$ and
$R_1$ is a finite $A$-module.
Let $M = \bigoplus_{n\in \ZZ}M_n$ be a graded $R$-module with $M_n = 0$ for all $n$ sufficiently small. We will assume throughout
that each $M_n$ is a finitely generated $A$-module. Note that we do \textit{not} assume that $M$
is a finitely generated $R$-module. Throughout, we set $L := H^0_{R_+}(M) = \bigoplus_{n\in \ZZ} L_n$.
We say that $M$ is  \emph{quasi-finite} if, in addition, $L_n = 0$ \ $\nsl$. It turns out that for the theory of 
asymptotic prime divisors, it is the quasi-finite property that is crucial, and not the finite generation of the module $M$. Indeed, almost all of the results we give in the context of graded modules are already known in the finite case (see \cite{M}, \cite{EW} and \cite{H}). Of course, finite modules are quasi-finite. For an example of a quasi-finite module  which is not finite, we will see below that $A[X]/A[JX]$ as a $A[JX]$-module, where $X$ is an indeterminate and
 $J\subseteq A$ is a proper 
ideal with positive grade. The notion of quasi-finite modules  was introduced in \cite{JPV}.

It turns out that our results are not much harder to come by if we consider multigraded rings, i.e., standard  $\Nd$-graded Noetherian $A$-algebras and multigraded modules over them. In section two, we 
define the types of modules we are interested in and, in particular, we extend the definition of quasi-finite module to the multigraded case. We then prove our basic results concerning asymptotic prime divisors of quasi-finite multigraded modules. In particular, we note that we can achieve the standard stability result known for finite modules (see Theorem \ref{main}). In section three we give some specific examples and applications of the results in section two, especially in the singly graded case. Our problems would be simpler to solve but less interesting if the graded (or multigraded) modules under consideration were always quasi-finite. To deal with modules $M$ 
which are not quasi-finite, we look at the quasi-finite module $M/L$. In section four, we consider what happens in the case of 
multigraded modules that are not necessarily quasi-finite. It turns out that we can isolate the precise obstruction to stability of asymptotic prime divisors in the general case (see Theorem \ref{genstability}). 
Finally, in section five we present a multi-ideal version of the result alluded to in the abstract. In particular, we prove the following theorem.

\begin{theorem}\label{introthm} Let $A$ be a noetherian ring, $C$ a finitely generated $R$-module and $\J \subseteq A$ finitely many ideals. Suppose that for 
each $1\leq i\leq d$, $\{I_{i,n_i}\}_{n_i\geq 0}$ is a filtration of ideals satisfying : (i) $I_{i,0} = A$, 
(ii) $J_i^{n_i}\subseteq I_{i,n_i}$, for all $n_i\geq 0$, (iii) For all $m_i\leq n_i\in \NN$, $I_{i,n_i}\subseteq I_{i,m_i}$ and (iv) $J_{r_i}\cdot I_{i,s_i} \subseteq I_{i,r_i+s_i}$, for all $r_i$ and $s_i$. Then there exists $k = (k_1, \ldots, k_d) \in \Nd$ such that for all $n = (n_1, \ldots, n_d)\geq k$, \[\Ass_A(I_{1,n_1}\cdots I_{d,n_d}C/J_1^{n_1} \cdots 
J_d^{n_d}C) = \Ass_A(I_{1,k_1}\cdots I_{d,k_d}C/J_1^{k_1} \cdots 
J_d^{k_d}C).\]
\end{theorem}

\section{Quasi-finite multigraded modules}

Throughout this section $R$ will denote a standard $\Nd$-graded Noetherian, 
commutative ring with identity, where $\NN$ denote the set of non-negative integers. We denote the degree $(0,\ldots, 0)$ component of $R$ by $A$. Here, we use the term `standard' 
in the sense of Stanley, i.e., a standard
$\Nd$-graded ring is one which is generated in total degree one. Rather than use excessive notation, we will simplify our notation 
and use $n\in \Nd$ to indicate $d$-tuples. We will use subscripts to 
denote components of $d$-tuples. Thus 
$n_i$ means the $i^{\textrm{th}}$ component of $n = (n_1,\ldots, n_d) \in 
\Nd$. Superscripts will be used to indicate lists of $d$-tuples. Given $n,m\in \Nd$, we will write 
$n\geq m$ if $n_i\geq m_i$, for all $1\leq i\leq d$. Finally, we extend all of this notation in the obvious way to $\ZZ ^d$.

\s \textbf{Notation.} Let $R = \bigoplus_{n\in \Nd}R_d$ be a Noetherian standard $\Nd$-graded ring as above.

\begin{enumerate}
\item[(a)] We will write 
$R_+$ for the ideal consisting of all sums of homogeneous elements $x_n \in R_n$ such that $n_i\geq 1$, for all 
$1\leq i\leq d$. In other words, $R_+$ denotes the ideal of $R$ generated by 
$R_{(1,\ldots, 1)}$. 

\item[(b)] \textit{Throughout this paper}, by a multigraded $R$-module we mean a $\ZZ ^d$-graded $R$-module $M = \bigoplus_{n\in \ZZ ^d} M_n$ such that : 

\begin{enumerate}
\item[(i)] Each component $M_n$ of $M$ is a finitely generated $A$-module. 
\item[(ii)] There exists $a \in \ZZ ^d$ such that $M_n = 0$ for $n \leq a$.
\end{enumerate}

\end{enumerate}

\medskip
\noindent
\s \textbf{Observation.}\label{observation} In the notation above, suppose that $x\in M$ and $R_c\cdot x = 0$, for 
some $c\in \Nd$ with $c \not = (0,\ldots,0)$. Then $x\in H^0_{R_+}(M)$. To see this, suppose that 
$t\in \NN$ is the largest component of $c$, so $t > 0$. Then $R_{(t-c_1,\ldots, t-c_d)}\cdot 
R_c\cdot x = 0$. In other words, $R_{(t,\ldots, t)}\cdot x = 0$. Since $R$ is standard graded, 
$R_+^t \cdot x = 0$, therefore $x\in H^0_{R_+}(M)$.  

We now define multigraded quasi-finite modules.

\begin{definition}\label{quasfinitedef}
Let $M = \bigoplus_{n\in \ZZ ^d} M_n$ as above be an multigraded $R$-module. We say $M$ is a \emph{quasi-finite} $R$-module if there exists $b \in \ZZ ^d$ such that $H^0_{R_+}(M)_n = 0$ for $n \geq b$.
\end{definition}

\begin{remark} \label{multiquasiexample} Notice that if $M$ is an $\ZZ ^d$-graded $R$-module as above, then it follows immediately from the definition that $M/H^0_{R_+}(M)$ is quasi-finite as a $\ZZ ^d$-graded $R$-module. More generally, the next propostion shows that a wide range of multigraded modules are quasi-finite. 
\end{remark}

\begin{proposition}\label{quasigradepositive} Let $U\subseteq V$ be multigraded $R$-modules such that $U$ is finitely generated over $R$ and $\grade(R_+,V) > 0$. Then $M := U/V$ is a quasi-finite $R$-module. 
\end{proposition}

\begin{proof}  Consider the exact sequence
\[
0 \rt U \rt V \rt M  \rt 0.
\]
Taking local cohomology \wrt \ $R_+$ and using that $H^0_{R_+}(V) = 0$ we get
\[
0 \rt H^0_{R_+}(M ) \rt H^1_{R_+}(U).
\]
But, $U$ is a finitely generated $R$-module, so there exists $b \in \ZZ ^d$ such that $H^1_{R_+}(U)_n = 0$, for $n\in \Nd$ with $n\geq b$. It follows that $M$ is quasi-finite.
\end{proof}

For example, if $\R$ is the Rees algebra determined by an ideal of positive grade and $V$ is the polynomial ring  containing $\R$, 
then, with $U = \R$, $M := V/U$ is an infinitely generated quasi-finite $\R$-module. (See section three.)

The following proposition is well known in the case of finitely generated graded or multigraded modules. Since $R$ is Noetherian, the proofs are the same even if $M$ is not finitely generated over $R$. 

\begin{proposition}\label{primesofM} Let $R$ be a not necessarily standard $\Nd$-graded Noetherian ring and $M$ as above be a multigraded $R$-module. Then $P\in \Ass_A(M)$ if and only if $P\in \Ass_A(M_n)$ for some $n\in \ZZ^d$.  Moreover,  $P \in \Ass_A(M_n)$ for some $n\in \ZZ ^d$ if and only if there exists
a prime $Q\in \Ass_R(M)$ with $Q\cap A = P$. Consequently, the following statements hold :
\begin{enumerate}

\item[(i)]  $\bigcup _{n\in \ZZ ^d} \Ass_A(M_n)$ is finite if and only if $\Ass_A(M)$ is finite. 
\item[(ii)] If $\Ass_R(M)$ is finite, then $\bigcup _{n\in \ZZ ^d} \Ass_A(M_n)$ is finite.
\end{enumerate}
\end{proposition}

Note also that for any multiplicatively closed subset $S\subseteq A$, $R_S$ is a standard $\Nd$-graded Noetherian $A_S$-algebra and $M_S$ is a multigraded $R_S$-module with the original gradings preserved since the elements of $S$ have degree zero. It follows easily from this that if 
$P\subseteq A$ is disjoint from $S$ then $P$ is the annihilator of an element of degree $n$ in $M$ if and only if $P_S$ is the annhilator of an element of degree $n$ in $M_S$. We will use this observation freely throughout this paper.

Our first theorem shows that the quasi-finite notion is sufficient to guarantee asymptotic stability of prime divisors. 
\begin{theorem}\label{main} Let $R$ be a standard $\Nd$-graded Noetherian ring and $M$ as above be a quasi-finite $R$-module. Then there exists $b\in \ZZ ^d$ such that for all $n, m\in \ZZ ^d$ with 
$b\leq n\leq m$, $\Ass _A(M_n) \subseteq \Ass_A(M_m)$. Moreover, if $\Ass_A (M)$ is a finite set, then there exists $k\in \ZZ ^d$ such that for all $n \in \ZZ ^d$ with $n\geq k$, 
$\Ass_A( M_n) = \Ass_A (M_k)$. 
\end{theorem}
\begin{proof}
 Now, let $b\in \ZZ ^d$ be such that $H^0_{R_+}(M)_b = 0 $ for all $n\in \ZZ ^d$ with $n\geq b$.  If $M_n = 0$ for all $n\in \ZZ^d$ with $n\geq b$, then $\Ass_A(M_n) = \emptyset$ for all such $n$ and the conclusions of the theorem readily follow. Otherwise, $M_n \not = 0$ for some $n\geq b$. Without loss of generality, we may take $n = b$ (by increasing $b$ if necessary) and assume that $M_b \not = 0$. Note that since $M_b\not \subseteq H^0_{R_+}(M)$, it follows from Observation \ref{observation} that for all $n\geq b$, $R_{n-b}\cdot M_b \not = 0$. In particular, $M_n \not = 0$.  

Now, take $c < h\in \Nd$. We first show that $\Ass_A (M_{b+c}) \subseteq \Ass_A (M_{b+h})$. Let 
$P \in \Ass _A (M_{b+c})$. Without loss generality, we may assume that $A$ is local at $P$ and 
$P = (0:u)$, for $0 \not = u \in M_{b+c}$. We now note that $R_{h-c}\cdot u \not = 0$. Indeed, suppose 
$R_{h-c}\cdot u = 0$. Then by Observation \ref{observation} above, $u \in H^0_{R_+}(M)$. Since $u \not = 0$, this contradicts our choice of $b$. Thus, $R_{h-c}\cdot u \not = 0$. Therefore, $xu \not = 0$, for some $x\in R_{h-c}$. Since $P\cdot xu =0$, $P\in \Ass_A (M_{b+h})$, as required. 

Now, suppose $\Ass_A(M)$ is finite. Then by Proposition \ref{primesofM}, $\bigcup_{n\in \ZZ ^d} \Ass_A(M_n)$ is a finite set.
Now, let $\{P_1, \ldots, P_r\}$ denote the prime ideals $\bigcup_{n\geq b}\Ass_A (M_n)$. We can write 
each $P_j = (0:_Au_j)$, where $u_j \in M_{h^j}$, for $h^1, \ldots, h^r \in \ZZ ^d$, with each $h^j\geq b$.  Choose $k\in \ZZ ^d$ such that $h^j\leq k$, for all $1\leq j\leq r$. Then, by the paragraph above, $P_j\in \Ass_A (M_k)$ for all 
$j$ and hence, $P_j\in \Ass_A M_n$, for all $k\leq n$. On the other hand, if $n\geq k$, then $n\geq b$. Thus, if 
$P\in \Ass_A (M_n)$, $P = P_j$, for some $1\leq j\leq r$. Thus, for all $n\in \ZZ ^d$ with $n\geq k$, 
$\Ass_A(M_n) = \{P_1,\ldots, P_r\} = \Ass_A (M_k)$, and this completes the proof of the proposition. \end{proof}

\section{First applications}
Suppose we have a homogeneous inclusion of singly graded Noetherian $R$-algebras $R \subseteq S$. In other words, $S = \bigoplus_{n\geq 0} S_n$ is a Noetherian ring
with $S_0 = R_0 = A$ and $S_n \supseteq R_n$ for all $n \geq 0$. We assume $R$ is standard graded, but \textbf{do not} assume that $S$ is standard graded.
Consider the $R$-module $E = S/R$. Notice that
\[
E= \bigoplus_{n = 0}^{\infty} \frac{S_n}{R_n} \quad \text{as an $A$-module}.
\]
Note that $E$ need not be a finitely generated $R$-module. We however have the following :
\begin{lemma}
\label{Rees}
$\Ass_A (E)$ is a finite set.
\end{lemma}
\begin{proof}
We use a technique due to Rees \cite{Rees1}. Let $I$ denote the ideal of $S$ generated by $R_1$. Then, as graded ideals in $S$, for each $p\geq 0$, $I^p = \bigoplus_{n\geq 0}R_pS_n$ and $I^{p+1} = 
\bigoplus_{n\geq 0} R_{p+1}S_n$. Thus, as a graded $S$-module, $I^p/I^{p+1} = 
\bigoplus_{n\geq 0} R_pS_n/R_{p+1}S_{n-1}$. If we now write $G$ for the associated graded ring of $S$ with respect to $I$, we then have \[G = \bigoplus_{n,p\geq 0}R_pS_n/R_{p+1}S_{n-1},\]
and as such, $G$ can be viewed as a finitely generated (not necessarily standard) bigraded $A$-algebra. Thus, 
$\Ass_G(G)$ is finite, so by Proposition \ref{primesofM}, $\Ass_A(G)$ is finite. 

The advantage of  Rees's technique is due to the following observation. For $t\geq 1$, consider the filtration of   $A$-modules
\[
 R_t \subseteq  R_{t-1} S_1 \subseteq R_{t-2} S_2\subseteq \ldots \subseteq  R_{1}S_{t-1} \subseteq S_{t}.
 \]
By breaking this filtration into short exact sequences, it follows that 

\[ \Ass_A(S_t/R_t) \subseteq \bigcup_ {j =0}^t \Ass_A(R_{t-j}S_j/R_{t-j+1}S_{j-1}) \subseteq \Ass_A(G).\]
It follows that $\bigcup_{t\geq 1} \Ass_A(R_t/S_t)$ is finite, which gives what we want.
\end{proof}
\noindent
The next proposition is just a variation on Proposition \ref{multiquasiexample} and indicates when $E$ as above is quasi-finite.
\begin{proposition}\label{criterion} For $R, S$ and $E$ as above, the following are equivalent
\begin{enumerate}
\item
$E = S/R$ is a quasi-finite $R$-module.
\item
$S$ is a quasi-finite $R$-module.
\end{enumerate}
In particular if $\grade(R_+, S) > 0$ then $S/R$ is a quasi-finite module.
\end{proposition}
\begin{proof}
This follows from taking
local cohomology (with respect to $R_+$) of the short exact sequence 
\[
0 \xar R \xar S \xar S/R \xar 0
\]
and noting that for each $i \geq 0$ we have $H^i_{R_+}(R)_j = 0$ for all $j \gg 0$.
\end{proof}

The following corollary is an immediate consequence of Theorem \ref{main}, Lemma \ref{Rees}, and Proposition \ref{criterion}. It recovers from our perspective a special case of Theorem 1.1 from \cite{H}, though in our case, we do not need to assume that $S$ is standard graded. 

\begin{corollary}\label{stablegraded} Let $R\subseteq S$ be as in the previous proposition. If $\grade(R_+,S) > 0$, then 
$\Ass_A(S_n/R_n)$ is independent of $n$, for $n$ sufficiently large.
\end{corollary}

\begin{remark}\label{hpvexample} It should be noted that the conclusion of Corollary \ref{stablegraded} can fail if $\grade(R_+,S) = 0$. Indeed, in \cite{JPV}, Example 3.4, one has $S_n/R_n = 0$ for $n$ odd and $S_n/R_n \not = 0$ for $n$ even. Thus, for this example, $\Ass_A(S_n/R_n)$ is not stable for $n$ large. 
\end{remark}

In the example below, we use a result due to Herzog, Hibi and Trung from their paper \cite{HHT}.

\begin{example}
Let $A = K[X_1,\ldots,X_d]$ be a polynomial ring over a field $K$. Let $I_1, \ldots, I_r$ be monomial ideals in $A$. Then there exists $k\in \NN$ such that for all $n\geq k$,
\[
\Ass_A \frac{ \bigcap_{i=1}^{r} I_i^n}{\left(\bigcap_{i=1}^{r} I_i \right)^n} = \Ass_A \frac{ \bigcap_{i=1}^{r} I_i^{k}}{\left(\bigcap_{i=1}^{r} I_i \right)^{k}}.
\]
\end{example}
\begin{proof}
Consider the algebra $S = \bigoplus_{n \geq 0} \bigcap_{i=1}^{r} I_i^n$. By \cite[Corollary 1.3]{HHT} we get that $S$ is a finitely generated $A$-algebra. Notice that $R = \bigoplus_{n \geq 0}\left(\bigcap_{i=1}^{r} I_i \right)^n$ is a standard graded $A$-subalgebra of $S$.
Since $S$ is a domain we have that $\grade(R_+, S) > 0$. The result now follows by Corollary \ref{stablegraded}. 
\end{proof}

We continue with our applications in the singly graded case, by way of illustrating the strength of the quasi-finite condition. We begin by letting $J\subseteq A$ be an ideal and $\{ I_n \}_{n\geq 0}$ a filtration of ideals in $A$ satisfying the following properties : $I_0 = A$, $J^n\subseteq I_n$ for all $n \geq 0$ and $J_r\cdot I_s \subseteq I_{r+s}$, for
all $r$ and $s$. We set $\R := \bigoplus_{n\geq 0} J^n$, the Rees algebra of $A$ with respect to $J$. For a
finitely generated $A$-module $C$, we write $U:= \bigoplus_{n\geq 0} I_nC$, a not necessarily
finite $\R$-module. Let $V := \bigoplus_{n\geq 0}J^nC$ be the Rees module of $C$ with respect to $J$. We set $M_C := U/V$. Notice that since
$\bigcup_{n \geq 1} \Ass_A (C/J^nC)$ is a finite set, it follows that $\Ass_A M_C$ is a finite set. The following application is a special case of our main result in section five. 

\begin{proposition}\label{gradePositive} Maintain the notation established in the paragraph above. Suppose 
$\grade(J,C) > 0$. Then $M_C$ is quasi-finite. In particular, $\Ass_A (I_nC/J^nC)$ is stable for all $n \gg 0$.
\end{proposition}
\begin{proof}
If $\grade(J,C) > 0$, then $\grade(\R_+, V) > 0$, so $M_C$ is a quasi-finite $\R$-module, by Proposition \ref{quasigradepositive}. For the second statement, $\Ass_A (M_C)$ is a finite set, by our comment in the previous paragraph, so the result now follows from Theorem \ref{main}.
\end{proof}

\section{The non-quasi-finite case}
We now return to the notation of section two. That is, $R$ denotes a Noetherian standard $\Nd$-graded $A$-algebra and $M$ is a (not necessarily finitely generated) multigraded $R$-module satisfying our standard hypotheses. Throughout we set $L := H^0_{R_+}(M)$. Notice that $M/L$ is quasi-finite,  since $H^{0}_{R_+}(M/L) = 0$.

We begin with the following proposition.

\begin{proposition}\label{union} Let $R$ be a Noetherian standard $\Nd$-graded $A$-algebra and $M$ a multigraded
$R$-module. Then,
$$\Ass_A (M) =  \Ass_A (L) \cup \Ass_A (M/L).$$
In particular, for any $q\in \ZZ ^d$, if $M_q/L_q \not = 0$, there exists $s\in \Nd$ such that $\Ass_A(M_q/L_q) \subseteq 
\Ass_A(M_{q+s})$. Moreover, if $\Ass_A M $ is a finite set, then there exists $k \in \Nd$ such that $\Ass_A(M_n/L_n) = \Ass_A(M_k/L_k)$ for all $n\in \Nd$ with $n\geq k$.
\end{proposition}
\begin{proof} Because $\Ass_A(M)\subseteq \Ass_A(L)\cup \Ass_A(M/L)$, for the first statement it suffices to 
prove that $\Ass_A(M/L)$ is contained in $\Ass_A(M)$. But for this, the second statement suffices. Let $P\in \Ass_A(M_q/L_q)$. Then we may write $P = (L_q :_A u)$, for 
$u\in M_q\backslash L_q$. Thus, $P\cdot u\subseteq L_q$, so there exists $t > 0$ such that $R_{_+}^t \cdot Pu = 0$. In particular, $P\cdot R_su = 0$, for $s = (t,\ldots, t)$. Thus, $P\subseteq (0 :_A R_su)$. Suppose 
$a\in A$ and $a\cdot R_s u = 0$. Then, $R_{_+}^t\cdot au = 0$, so $au \in L_q$. Thus, $a \in P$. 
It follows that $P = (0:_AR_su)$. Since $R_su$ is a finite $A$-module, it follows that $P$ annihilates an element of $R_su$, and hence $P\in \Ass_A(M_{q+s})$. Since we may choose the 
same $t$ for all $P\in \Ass_A(M_q/L_q)$ (since $M_q/L_q$ is a finite $A$-module), it follows that 
$\Ass_A(M_q/L_q) \subseteq \Ass_A(M_{q+s})$. 
This proves our second assertion. For the final statement, if $\Ass_A (M) $ is a finite set, it follows from the first statement of this proposition that $\Ass_A (M/L) $ is a finite set. However, $M/L$ is quasi-finite, so our last assertion follows from
Theorem \ref{main}.
\end{proof}

\begin{remark}\label{multiremark} In the preceding proof, note that since $u\not \in L$, $R_+^t\cdot u \not \subseteq L$, so that if 
$P\in \Ass(M/L)$, $P$ is the annihilator of an element of $M\backslash L$. 
\end{remark}

We now set $L':= (0 \colon_{M} R_+)$. Clearly $\Ass_A (L') \subseteq \Ass_A (L)$. The following proposition supplies the reverse containment.

\begin{proposition}\label{L'equalsL} If $P \in \Ass_A (L_n)$ for some $n\in \Nd$, then $P \in \Ass_A (L'_{n+s})$ for some $s \in \Nd$. In particular, $\Ass_A(L') = \Ass_A(L)$. 
\end{proposition}
\begin{proof} Only the first statement requires proof. 
Let $P \in \Ass_A (L_n)$. By localizing at $P$ we may assume that $A$ is local
and $P  = (0 \colon u)$ where $u \in L_n$. Now, for some $t\geq 1$, $R_+^t\cdot u=0$.
If $t=1$, then $u \in L'_n$ and $P \in \Ass(L'_n)$, as required. Otherwise, we choose
$t > 1$ least so that $R_+^{t-1}\cdot u \neq 0$.
Notice $P\cdot R_+^{t-1}\cdot u = 0$. Set $s := (t-1,\ldots, t-1)$. Then there is $x \in R_s$ be such that $x\cdot u\not = 0$.
We have $P\cdot xu = 0$. Also $R_+ \cdot(x u) = 0$. Thus, $x u \in L'_{n+s}$.  It follows that $P  \in \Ass_A (L'_{n+s})$. 
\end{proof}

Here is the main result of this section. It illustrates an obstruction to the asymptotic stability of 
$\Ass_A(M_n)$ in case $M$ is not quasi-finite. 

\begin{theorem}\label{genstability} Let $R$ as above be a Noetherian standard $\Nd$-graded $A$-algebra and $M$ a multigraded $R$-module as above. Assume that $\Ass_A(M)$ is finite. Consider the following statements : 
\begin{enumerate}
\item[(i)] There exists $k\in \ZZ ^d$ such that $\Ass_A(L'_n) = \Ass_A(L'_k)$, for all $n\in \ZZ ^d$ with $n\geq k$. 
\item[(ii)] There exists $l \in \ZZ ^d$ such that $\Ass_A(L_n) = \Ass_A(L_l)$, for all $n\in \ZZ ^d$ with $n\geq l$. 
\item[(iii)] There exists $h\in \ZZ ^d$ such that $\Ass_A(M_n) = \Ass_A(M_h)$, for all $n\in \ZZ ^d$  with $n\geq h$.
\end{enumerate}
Then (i) implies (ii) and (ii) implies (iii). 
\end{theorem}

\begin{proof}  We first note that if there exists $t\in \ZZ^d$ with $L'_n = 0$ for all $n\geq t$, then $L_n = 0$, for all $n\geq t$. Indeed, if some $L_n \not = 0$ for $n\geq t$, then there exists $P\in \Ass_A(L_n)$. But by Proposition \ref{L'equalsL}, $P\in \Ass_A(L'_{n+s})$, for some $s\in \Nd$, and therefore, $L'_{n+s} \not = 0$, a contradiction. Thus, if $\Ass_A(L'_n) = \emptyset$, for all $n\geq k$, then the same holds for $\Ass_A(L_n)$.  Now suppose $\emptyset \not = \Ass_A  L'_{n}  = \Ass_A(L'_k)$, for all 
$n\geq k$. Fix $n \in \Nd$ with $n\geq k$. Take $P\in \Ass_A(L_k)$. Then, by Proposition 
\ref{L'equalsL}, there exists $s\in \Ns$ with $P\in \Ass_A(L'_{k+s})$. Thus, $P\in \Ass_A (L'_k)$, by choice of $k$. Thus, $P\in \Ass_A(L'_n)\subseteq \Ass_A(L_n)$. Conversely, if $P\in \Ass_A(L_n)$, then 
$P\in \Ass_A(L'_{n+r})$, for some $r\in \Nd$. Thus, $P\in \Ass_A(L'_k)\subseteq \Ass_A(L_k)$. Thus, 
$\Ass_A(L_n) = \Ass_A(L_k)$, for all $n\in \Nd$ with $n\geq k$. Therefore (i) implies (ii).

Now suppose that (ii) holds. We have two cases. If $\Ass_A(L_n) = \emptyset$ for all $n\geq l$, then $L_n = 0$, for all $n\geq l$, and therefore $M$ is quasi-finite. Thus, the conclusion of (iii) follows from Theorem \ref{main}. If $\Ass_A(L_n) \not = \emptyset$ for all $n\geq l$, then $L_n \not = 0$, for all $n\geq l$, and thus, $M_n\not = 0$, for all $n\geq l$. 

To continue, first note that since $\Ass_A(M)$ is finite, Proposition \ref{union}  implies that $\Ass_A(M/L)$ is also finite. By Theorem \ref{main}, 
there exists $q\in \ZZ ^d$ such that $\Ass_A (M/L)_n = \Ass_A(M/L)_q$ for all $n\in \ZZ ^d$ with $n\geq q$, since 
$M/L$ is quasi-finite. If the stable value of $\Ass_A(M_n/L_n) = \emptyset$, then $M_n = L_n$ for all $n\geq q$, and thus $\Ass_A(M_n) = \Ass_A(M_l)$, for all $n\in \Nd$ with $n\geq h$, for any $h\in \ZZ^d$ with $h\geq q$ and $h\geq l$, which gives what we want. Otherwise, if $M_q/L_q \not = 0$, then by Proposition \ref{union}, there exists $q'\in \ZZ ^d$ with $q' \geq q$ such that 
$\Ass_A( M_q/L_q)\subseteq \Ass_A (M_{q'})$. 

Take $h\in \ZZ ^d$ such that $h\geq l$ and $h\geq q'$. Let $P \in \Ass_A(M_h)$. We will show $P\in \Ass_A(M_n)$ for all $n\geq h$, and for this, we may assume that $A$ is local at $P$. If $P\in \Ass_A(L_h)$, then by the choice of $l$, $P \in \Ass_A(L_n)$ for all $n\in \ZZ ^d$ with $n\geq h$. Thus $P\in \Ass_A(M_n)$ for all $n\geq h$. If 
$P\not \in \Ass_A(L_h)$, then $P = (0:_A u)$ for $u \in M_h\backslash L_h$. Thus, by Observation \ref{observation} above, 
$R_{n-h}\cdot u \not = 0$, for all $n\in \ZZ ^d$ with $n\geq h$. Note $n-h \geq 0$. Since 
$P$ annihilates $R_{n-h}\cdot u$ and $R_{n-h}\cdot u\subseteq M_n$, we have $P\in \Ass_A(M_n)$, for all $n\geq h$. 

Now suppose $n\in \ZZ ^d$ and $n\geq h$. Take $P\in \Ass_A(M_n)$. If $P\in \Ass_A(L_n)$, then $P\in \Ass_A(L_h)$, by our choice of $h$ and thus, $P\in \Ass_A(M_h)$. Otherwise, 
$P\in \Ass_A(M_n/L_n)$. Thus, $P\in \Ass_A(M_q/L_q)$, by the definition of $q$. Therefore, by the definition of $q'$, $P\in \Ass_A(M_{q'})$, and by Remark \ref{multiremark} above, we may assume that $P$ is the annihilator of an 
element of $M_{q'}\backslash L_{q'}$. Thus, as in the second paragraph of this proof, we may move $P$ forward so that $P\in 
\Ass_A(M_h)$, since $h\geq q'$. We now have $\Ass_A(M_h) = \Ass_A(M_n)$, for all $n\geq h$, which 
gives (iii). 
\end{proof}

\begin{remark} (a) Clearly, in light of Proposition \ref{primesofM}, even in the singly graded case we cannot expect stability of
$\Ass_A (M_n)$ without a finiteness condition on $\Ass_A (M)$. Moreover, in the presence of
this condition, we cannot do better than Theorem \ref{genstability}. Indeed, notice that $L'$ is 
just a direct sum of $A$-modules. Thus, for example, let $\gamma$ be
an irrational number with a binary expansion $\gamma := a_1a_2a_3\cdots$ and take two distinct primes
$P _1$ and $P _2$ in $A$. If we set $R = A$ and $M := \bigoplus_{n\geq 0} M_n$, where
$M_n := A/P _1$, for $a_n=0$ and $M_n := A/P _2$, for $a_n=1$, then $L' = M$ and
$\Ass_A (M)$ is finite, but $\Ass_A (M_n)$ is neither stable nor periodic. 

\medskip
\noindent
(b) Neither of the implications in Theorem \ref{genstability} can be reversed. To see this, let 
$R = k[X]$ denote the polynomial ring in one variable over the field $k$. Let $T$ be the graded $R$-module $R/X^2R$. Thus,  $T = T_0\oplus T_1$, with $T_0$ and $T_1$ one-dimensional vector spaces over $k$. Moreover, $X^2\cdot T _0 = 0$, $X\cdot T_0 \not = 0$ and $X\cdot T_1 = 0$. Let 
$M = \bigoplus_{n\geq 0} M_n$, where $M_n = T_0$ if $n$ is even and $M_n = T_1$ if $n$ is odd. 
Then $M = L := H^0_{R_+}(M)$ and $\Ass_k(M_n) = (0)$, for all $n$, so $\Ass_A(L_n)$ is stable for 
all $n$. On the other hand, setting $L' := (0:_R R_+)$, it follows that $L'_n = 0$ for $n$ even and $L'_n \not = 0$, for $n$ odd. Thus, $\Ass_k(L'_n)$ is not stable for $n$ large. Therefore, statement (ii) in Theorem \ref{genstability} does not imply statement (i). A similar example can be constructed to show that (iii) does not imply (ii) in the statement of Theorem \ref{genstability}. 

\noindent
\end{remark}

\section{Second Applications}

 Let $C$ be a finitely generated $A$-module, $J_1, \ldots, J_d$ be a family of ideals and for 
each $1\leq i\leq d$ let $\{I_{i,n_i}\}_{n_i\geq 0}$ be a filtration of ideals satisfying : (i) $I_{i,0} = A$, 
(ii) $J_i^{n_i}\subseteq I_{i,n_i}$, for all $n_i\geq 0$, (iii) $I_{i,n_i}\subseteq I_{i,m_i}$, whenever $n_i\geq m_i$,  and (iv) $J_{r_i}\cdot I_{i,s_i} \subseteq I_{i,r_i+s_i}$, for all $r_i$ and $s_i$. An application of the main result of this section shows that if $C$ is a finitely generated $A$-module, there exists $k = (k_1, \ldots, k_d) \in \Nd$ such that for all $n = (n_1, \ldots, n_d)\geq k$, \[\Ass_A(I_{1,n_1}\cdots I_{d,n_d}C/J_1^{n_1} \cdots 
J_d^{n_d}C) = \Ass_A(I_{1,k_1}\cdots I_{d,k_d}C/J_1^{k_1} \cdots 
J_d^{k_d}C).\]

Not suprisingly, we accomplish our goal by recasting the given data in terms of multigraded rings and modules. Throughout the rest of this section, we fix ideals $J_1, \ldots, J_d \subseteq A$. Keeping the notational conventions from the previous section, we will write $J^n$ for 
$J_1^{n_1}\cdots J_d^{n_d}$ for all $n = (n_1, \ldots, n_d) \in \Nd$. 

We will need the following definition.

\begin{definition}\label{multifil} For $J_1, \ldots, J_d$ as above, we call a collection of ideals 
$\{I_n\}_{n\in \Nd}$ a \textit{multi-filtration} with respect to $J_1, \ldots, J_d$ if the following conditions holds :
\begin{enumerate}
\item[(i)] $I_{(0,\ldots,0)} = A$.
\item[(ii)] $J^n\subseteq I_n$, for all $n\in \Nd$.
\item[(iii)] For all $n\leq m\in \Nd$, $I_m\subseteq I_n$.
\item[(iv)] For all $n, h\in \Nd$, $J^n\cdot I_h \subseteq I_{n+h}$.
\end{enumerate}
\end{definition}

Given a multi-filtration with respect to $J_1, \ldots, J_d$, it is now a simple matter to create a multigraded set-up to which we can apply the results of the previous section.

\s\textbf{Notation.} Let $\{I_n\}_{n\in \Nd}$ be a multi-filtration with respect to $\J$. Let $C$ be a finitely generated $A$-module. We set $\R := \bigoplus _{n\in \Nd} J^n$, the multigraded Rees ring of $A$ with respect to 
$\J$. We also set $U :=\bigoplus_{n \in \Nd} I_nC$, $V := \bigoplus_{n\in \Nd} J^nC$ and $M := U/V$. Note that 
$\R$ is a standard $\Nd$-graded Noetherian ring, and $U, V$ and $M$ are multigraded $\R$-modules 
satisfying the standard hypotheses from the previous section. 

Our goal now is to show that, with the notation just established, there exists $k\in \Nd$ such that for all $n\in \Nd$ with 
$n\geq k$, $\Ass_A (M_n) = \Ass_A (M_k)$. We first note that $\Ass_A (M)$ is finite, since on the one hand, 
$M\subseteq \bigoplus_{n\in \Nd} C/J^nC$ while on the other hand, it is well known that 
$\bigcup_{n\in \Nd} \Ass_A (C/J^nC)$ is finite (in fact, ultimately stable), and therefore, 
$\Ass _A   (\bigoplus_{n\in \Nd}  C/J^nC)$ is finite. 

We need the following lemma which follows from \cite{KS}, Lemma 1.3 (see also \cite{KMR}, Proposition 1.4).

\begin{lemma}\label{cancellation} Let $\J \subseteq A$ be as above and suppose that $C$ is a finitely generated $A$-module. Suppose that each $J_i$ contains a nonzero divisor on $C$. Then 
there exists $k\in \Nd$ such that for all $n\in \Nd$ with $n \geq k$, $(J^{n+r}C :_C J^r) = J^nC$, 
for all $r\in \Nd$.
\end{lemma}

We are now ready for the principal result of this section, which immediately yields the main result of this paper.

\begin{theorem}\label{multimainthm} Let $A$ be a Noetherian ring, $C$ a finitely generated $A$-module and $\J\subseteq A$ finitely many ideals. Suppose that $\{I_n\}_{n\in \Nd}$ is a multi-filtration with respect to $\J$. Then there exists $k\in \Nd$ such that for all $n\in \Nd$ with $n\geq k$, $\Ass_A (I_nC/J^nC) = \Ass_A(I_kC/J_kC)$. 
\end{theorem}

\begin{proof} Set $\R := \bigoplus_{n\in Nd} J^n$ and $M := \bigoplus_{n\in \Nd} I_nC/J^nC$, so that 
$M$ is a not necessarily finitely generated multigraded $\R$-module. Following the notation of the previous section, we set $L := H^0_{R_+}(M)$.  Since $\Ass_A(M)$ is finite, by Theorem \ref{genstability} it suffices to show that there exists $l\in \Nd$ such that $\Ass_A(L_n) = \Ass_A(L_l)$, for all $n \in \Nd$ with $n\geq l$.
 
We now calculate an expression for $L_n$, for $n\in \Nd$. Suppose $u\in L_n$. Then there exists 
$t\in \NN$ such that $R_{_+}^t\cdot u = 0$ in $L$, for all $u\in L_n$ (since $L_n$ is a finitely generated 
$A$-module). It follows that if we set $r = (t,\ldots, t) \in \Nd$, then $J^r\cdot u\subseteq J^{n+r}C$ in $C$. Thus, $L_n = (J^{n+r}C :_C J^r)\cap I_n/J^n$. Note, $r$ depends upon $n$. 
Now, set $T := H^0_J(C)$, where $J = J_1\cdots J_d$. Thus, each $J_j$ contains a non-zerodivisor on $C/T$. It follows from Lemma \ref{cancellation} that there exists $k\in \Nd$ such that for all 
$n\in \Nd$ with $n\geq k$ and all $r\in \Nd$, $(J^{n+r}C :_C J^r) \subseteq  J^nC +T$. 

Now, suppose $n\in \Nd$ and $n\geq k$. We may choose $r$ as in the previous paragraph so that  \[T = (0 :_C J^r)\quad \textrm{and} \quad L_n = (J^{n+r}C :_C J^r)\cap I_n/J^n.\] In particular, $(J^{n+r}C :_C J^r) =  J^nC +T$. Increasing $k$ if necessary, we may further 
assume that $T\cap J^nC = 0$, by using the multigraded form of the Artin-Rees Lemma. Under these conditions 
\[L_n = (J^nC + T)\cap I_n/J^nC = (J^nC + I_nC\cap T)/J^nC = I_nC\cap T.\]
The path to the end of the proof is now clear. If we let $P_1, \ldots, P_g$ denote 
\[\bigcup_{n\geq k}\Ass_A(L_n) = \bigcup_{n\geq k} \Ass_A(I_nC\cap T),\] then each $P_i$ is an associated prime of some $L_{h^i}$
with $h^i\in \Nd$, $h_i\geq k$, $1\leq i\leq r$. We now just use the fact that for all $m,n\in \Nd$ with $m\geq n\geq k$, 
$L_m\subseteq L_n$. Now, if there exists $c\in \Nd$ such that $L_n = 0$ for all $n\geq c$, then 
we take $l = c$ and note that $\Ass_A(L_n) = \emptyset$, for all $n\geq c$. Otherwise, we order the set of primes 
$P_1, \ldots, P_g$ so that for each $P_1, \ldots, P_s$, if $P_i$ is among these primes, there exists 
$a^i\in \Nd$ with $P\not \in \Ass_A(L_{a^i})$ and no such $a^i$ exists if $s+1\leq i\leq g$. It follows that if $l\in \Nd$ and $l\geq a^i$ for $1\leq i\leq s$, then for all 
$n\geq l$, $P_i \not \in \Ass_A(L_n)$, if $1\leq i\leq s$ and $P_j\in \Ass_A(L_n)$, for $s+1\leq j\leq g$. 
In particular, $\Ass_A(L_n) = \Ass_A(L_l)$, for all $n\geq l$, and the proof of the theorem is complete.
\end{proof}

We now provide the result stated at the beginning of this section. 

\begin{corollary} Let $A$ be a noetherian ring, $C$ a finitely generated $R$-module and $\J \subseteq A$ finitely many ideals. Suppose that for 
each $1\leq i\leq d$, $\{I_{i,n_i}\}_{n_i\geq 0}$ is a filtration of ideals satisfying : 

\begin{enumerate}
\item[(i)] $I_{i,0} = A$. 
\item[(ii)] $J_i^{n_i}\subseteq I_{i,n_i}$, for all $n_i\geq 0$. 
\item[(iii)] $I_{i,n_i}\subseteq I_{i,m_i}$, for all $m_i\leq n_i\in \NN$.
\item[ (iv)] $J_{r_i}\cdot I_{i,s_i} \subseteq I_{i,r_i+s_i}$, for all $r_i$ and $s_i$. 
\end{enumerate}
Then there exists $k = (k_1, \ldots, k_d) \in \Nd$ such that for all $n = (n_1, \ldots, n_d)\geq k$, \[\Ass_A(I_{1,n_1}\cdots I_{d,n_d}C/J_1^{n_1} \cdots 
J_d^{n_d}C) = \Ass_A(I_{1,k_1}\cdots I_{d,k_d}C/J_1^{k_1} \cdots 
J_d^{k_d}C).\]
\end{corollary}

\begin{proof} For $n = (n_1, \ldots, n_d)$ in $\Nd$, set $I_n := I_{1,n_1}\cdots I_{d,n_d}$. The result follows immediately from Theorem \ref{multimainthm} since the 
collection of ideals $\{I_n\}_{n\in \Nd}$ is a multi-graded filtration with respect to $\J$.
\end{proof}

\begin{remark}\label{distinction} It is important to note that in the preceding theorem and corollary, we do not need to make any assumption that yields quasi-finiteness. This should be contrasted with Corollary \ref{stablegraded} and Remark \ref{hpvexample}. Indeed, if we take $d = 1$ in the preceding corollary, and set $R := \bigoplus_{n\geq 0} J^n$ and $S :=\bigoplus_{n\geq 0}I_n$, then we have $\Ass_A(S_n/R_n)$ is stable for large $n$, even if $\grade(R_+,S) = 0$.  
\end{remark}

We now give two examples illustrating certain aspects of the theorem. 

\begin{example}\label{symbolic-powers} Let $\J \subseteq A$ be finitely many ideals, let $K\subseteq A$ be any ideal and set
$I_n := J^n \colon K^\infty$, for all $n\in \Nd$. Then $\{J^n\colon K^{\infty}\}$ is a multi-filtration with respect to $\J$, so by Theorem \ref{multimainthm}, there exists $k\in \Nd$ such that $\Ass_A (J^n\colon K^{\infty})/J^n$ is stable for all $n \geq k$. Note, that in general, $\bigoplus_{n\in \Nd} (J^n\colon K^{\infty})$ need not 
be a Noetherian $A$-algebra. Note also that, as the proof of Theorem \ref{multimainthm} shows, the stability of $\Ass(M_n) = \Ass (J^n\colon K^{\infty})/J^n$ depends upon the stability of $\Ass(L_n)$, for $L := H^0_{R_+}(M)$. We determine this latter set of primes. It follows from the proof of Theorem \ref{multimainthm} 
that there exists $c \in \Nd$ such that $L_n = (J^n \colon K^\infty)\cap H^0_J(A)$, for all $n\in \Nd$ with $n\geq c$. (Recall $J = J_1\cdots J_d$.) We now note that 
\[
(J^n \colon K^\infty)\cap H^0_J(A) = H^0_K(A)\cap H^0_J(A), 
\]
for $n$ sufficiently large (in $\Nd$). To see this, clearly $H^0_K(A)\cap H^0_J(A)   \subseteq  (J^n \colon K^\infty)\cap H^0_J(A)$.
For the reverse inclusion, let $u \in (J^n \colon K^\infty)\cap H^0_J(A)$. Then $uK^r \subseteq J^n$ for some $r \geq 1$.
Notice $uK^r \subseteq H^0_J(A)$. So $uK^r \subseteq H^0_J(A)\cap J^n = 0$ for $n$ large in $\Nd$, by the multigraded version of the Artin-Rees lemma. Thus $u \in H^0_K(A)$, as required. It follows 
that in the present case, $\Ass (H^0_K(A)\cap H^0_J(A))$ is the stable value of $\Ass(L_n)$. Finally, suppose $H^0_J(A) \cap H^0_K(A) \not = 0$ and $P\in \Ass (H^0_J(A)\cap H^0_K(A))$. Then 
$P\in \Ass(A)$ and $P$ must contain both $J$ and $K$, else $H^0_K(A)_P\cap H^0_J(A)_P = 0$. Conversely, suppose $P\in \Ass(A)$ and $P$ contains $J$ and $K$. Then $P = (0\colon a)$ and thus $J\cdot a = 0$, so $a \in H^0_J(A)$. Similarly, $a\in H^0_K(A)$. It follows that $\Ass_A(H^0_J(A)\cap H^0_K(A)) = \Ass(A)\cap V(J+K)$.  Therefore, for all $n$ sufficiently large in $\Nd$, $\Ass(L_n) = \Ass(A)\cap V(J + K)$. 
\end{example}

\begin{example} Let $J\subseteq A$ be an ideal and take $I\subseteq A$ any ideal containing $J$ and set
$I_n := \ov{I^n}$. Then, by Theorem \ref{multimainthm}, $\Ass_A (I_n/J^n)$ is stable for all $n \gg 0$. 
Let $M = \bigoplus_{n\geq 1}\ov{I^n}/J^n$. Then by the proof of Theorem \ref{multimainthm}, $L_n = H^0_J(A) \cap \ov{I^n}$, for all $n\gg 0$, for $L = H^0_{R_+}(M)$. As in the previous example, we calculate the stable value 
of $\Ass_A(L_n)$. 
We now make two claims.

\medskip
\noindent
Claim 1 : $\Ass_A (H^0_J(A) \cap \ov{I^n})   =  \{ P \in \Ass_A \ov{I^n} \mid P \supseteq J \}$.

To see this, note that
if $P \nsupseteq J$ then $(H^0_J(A))_P = 0$, so $P \notin \Ass_A (H^0_J(A) \cap \ov{I^n})$. Thus,  $\Ass_A (H^0_J(A)\cap \ov{I^n})$ is contained in $\{ P \in \Ass_A \ov{I^n} \mid P \supseteq J \}$. On the other hand,
if $P \supseteq J$ and $P\in \Ass_A (\ov{I^n})$, then we localize at $P$ and assume $A$ is local at $P$ and $J \not = A$. Say $P = (0 \colon x)$ for $x \in \ov{I^n}$.
Then $P\cdot x = 0$.
So $x \in (0\colon P) \subseteq (0 \colon J^\infty)$. Thus $x \in  H^0_J(A) \cap \ov{I^n}$. Therefore $P \in \Ass_A (H^0_J(A) \cap \ov{I^n})$. This proves the
first claim.

Now, set $Z = $ nil-radical of $A$. Clearly we have the following exact sequence
 \begin{equation}
 0 \rt Z \rt \ov{I^n} \rt \ov{I^n}/Z \rt 0.
 \end{equation}
We make a second claim,

\medskip
\noindent
Claim 2 : For $n$ sufficiently large, $\Ass_A (\ov{I^n}) =  (\Ass_A (Z))\cup \{ P \mid P \in \Min(A) \ \textrm{and} \ I\nsubseteq P\}$. (Note, here we regard $Z$ as an $A$-module.)

To see, we first note that $\Ass_A (\ov{I^n}/Z) = \{ P \mid P \in \Min(A) \ \textrm{and} \ I\nsubseteq P\}$, for $n$ large. Clearly any minimal
prime $P$ not containing $I$ belongs to $\Ass_A (\ov{I^n}/Z)$ (just localize at $P$). On the other hand, if $P \in \Ass_A (\ov{I^n}/Z)$,
then $P \in \Ass_A (A/Z)$, so $P$ is a minimal prime. But if $P$ contains $I$, then $I_{P}$ is nilpotent. Thus,
$(\ov{I^n})_{P} = Z_{\fp}$ for $n$ large, so $\ov{I^n}_{P}/Z_{P} = 0$, Thus, $P \not \in \Ass_A (\ov{I^n}/Z)$, contradiction. Therefore, for $n$ sufficiently large,
\[\Ass_A (\ov{I^n}/Z) = \{ P \mid P \in \Min(A) \ \textrm{and} \ I\nsubseteq P\}.\] It follows immediately from this and (5.8.1) that
\[\Ass_A (\ov{I^n}) \subseteq (\Ass_A (Z))\cup \{ P \mid P \in \Min(A) \ \textrm{and}\ I\nsubseteq P\},\] for all large $n$. For the reverse containment,
clearly $\Ass_A (Z) \subseteq \Ass_A (\ov{I^n})$. Moreover, if $P \in \Min (A)$ and $P \nsupseteq I $ then by localizing at $P$ one sees that $P \in \Ass_A (\ov{I^n})$.
Thus, Claim 2 holds.

Now, as before, setting $L := H^0_{R_+}(M)$, we have $L_n = H^0_J(A)\cap \ov{I^n}$, for all large $n$ and by Claims 1 and 2,
\[ [(\Ass_A (Z))\cup \{ P \mid P \in \Min(A) \ \textrm{and} \ I\nsubseteq P\}]\cap V(J),\]
is the stable value of $\Ass_A(L_n)$, for $n$ sufficiently large.
\end{example}

\begin{remark}  Regarding Claim 2, we have an analogue for powers of an ideal. We note that for $n$ sufficiently large, 
 \[
 \Ass_A (I^n) = \{ P \mid P \in \Ass A \ \textrm{and} \ P \nsupseteq I \}.
 \]
Indeed, if $P \in \Ass_A (A)$ and $P \nsupseteq I$ then $I^n_P = R_P$. So $P \in \Ass_A (I^n)$ for all $ n\geq 1$.
On the other hand suppose $P \in  \Ass_A (I^n)$ and $P \supseteq I$, say $P = (0 \colon c_n)$ where $c_n \in I^n$.
So for $n$ sufficiently large we have \[c_n \in (0\colon P)\cap I^n = I^{n-n_0}(I^{n_0}\cap (0\colon P )) = 0,\] a contradiction. Thus, for sufficiently large $n$, if $P\in \Ass_A(I^n)$, $P\in \Ass_A(A)$, and $I\not \subseteq P$.
\end{remark}

Finally, we note the following, which should be clear from our results above. Let $C$ be a finitely generated $A$-module and $J_1, \ldots, J_d$ finitely many ideals. Assuming that $\bigcup_{n\in \Nd} C/J^nC$ is finite, asymptotic stability of $\Ass_A(C/J^nC)$, $n\in \Nd$ follows from Theorem \ref{multimainthm} by 
taking the multi-filtration with respect to $\J$ to be 
$I_n = A$ for all $n\in \Nd$. The finiteness of $\bigcup_{n\in \Nd} C/J^nC$  follows fairly readily from 
basic properties of prime divisors of regular elements on a module by using extended Rees algebras. For example, see \cite{KS}.

\end{document}